\documentclass[11pt,letterpaper]{amsart}
\pdfoutput=1

\usepackage{hyperref}
\usepackage{amsopn,amssymb,amsthm}
\usepackage{enumerate}
\usepackage{lpic}
\usepackage{caption}
\usepackage{setspace}
\usepackage{float}

\newcommand{\cC}{\mathcal{C}}
\newcommand{\cP}{\mathcal{P}}
\newcommand{\Mod}{\mathrm{Mod}}
\newcommand{\NCL}{\mathrm{NCL}}

\numberwithin{equation}{section}

\theoremstyle{plain}
\newtheorem{thm}{Theorem}
\newtheorem{cor}[thm]{Corollary}

\newtheorem{lemma}[thm]{Lemma}
\newtheorem{prop}[thm]{Proposition}

\theoremstyle{definition}
\newtheorem{definition}[thm]{Definition}
\newtheorem{example}[thm]{Example}
\newtheorem{question}[thm]{Question}

\newtheorem{remark}[thm]{Remark}

\newtheoremstyle{named}{}{}{\itshape}{}{\bfseries}{.}{.5em}{\thmnote{#3 }#1}
\theoremstyle{named}
\newtheorem*{namedtheorem}{Theorem}

\title[finite subgraphs of the curve graph]{on the complexity of finite \\subgraphs of the curve graph}
\author{Edgar A. Bering IV, Gabriel Conant, and Jonah Gaster}

\address{Department of Mathematics, Statistics, and Computer Science, 
University of Illinois at Chicago \\
Chicago, IL 60607}
\email{eberin2@uic.edu}

\address{Department of Mathematics, University of Notre Dame \\
Notre Dame, IN 46556}
\email{gconant@nd.edu}

\address{Department of Mathematics, Boston College \\ Chestnut Hill, MA 02467}
\email{gaster@bc.edu}

\begin{document}
\date{September 8, 2016}

\begin{abstract} 
We say a graph has property $\mathcal{P}_{g,p}$ when it is an induced subgraph of the curve graph of a surface of genus $g$ with $p$ punctures.
Two well-known graph invariants, the chromatic and clique numbers, can provide obstructions to $\cP_{g,p}$.
We introduce a new invariant of a graph, the \emph{nested complexity length},
which provides a novel obstruction to $\cP_{g,p}$.
For the curve graph this invariant captures the topological complexity of the surface in graph-theoretic terms; indeed we show that 
its value is $6g-6+2p$, i.e. twice the size of a maximal multicurve on the surface.
As a consequence we show that large `half-graphs' do not have $\cP_{g,p}$, and we deduce quantitatively that almost all finite graphs which pass the chromatic and clique tests do not have $\cP_{g,p}$. 
We also reinterpret our obstruction in terms of the first-order theory of the curve graph, and in terms of RAAG subgroups of the mapping class group (following Kim and Koberda).
Finally, we show that large multipartite subgraphs cannot have $\cP_{g,p}$. 
This allows us to compute the upper density of the curve graph, and to conclude that clique size, chromatic number, and nested complexity length are not sufficient to determine $\cP_{g,p}$.
\end{abstract}

\maketitle


\section{Statement of results}
\label{intro}

Let $S$ indicate a hyperbolizable surface of genus $g$ with $p$ punctures
(i.e.~$2g+p>2$). The \emph{curve graph} of $S$, denoted $\cC(S)$, is the infinite graph whose vertices are isotopy classes of simple closed curves on $S$ and whose edges are given by pairs of curves that can be realized disjointly. Let $\cP_{g,p}$ indicate the property that a graph is an induced subgraph of
the curve graph $\cC(S)$. 
We are concerned with the following motivating question:

\begin{question}
\label{motivation}
Which finite graphs have $\cP_{g,p}$? When is $\cP_{g,p}$ obstructed?
\end{question}

The low complexity cases $\cP_{0,3}$, $\cP_{0,4}$, and $\cP_{1,1}$ are trivial, so we assume further that $3g+p\ge 5$. See \S\ref{sec: notation} for details and complete definitions.

Property $\cP_{g,p}$ has been considered in different guises in the literature \cite{ehrlich-even-tarjan,crisp-wiest1,crisp-wiest2,koberda,kim-koberda1,kim-koberda2,kim-koberda3}. 
It is not hard to see that every finite graph has $\cP_{g,0}$ for large enough $g$, though it is remarkable that there exist finite graphs which do not have $\cP_{0,p}$ for any $p$ \cite[\S2]{ehrlich-even-tarjan}.\footnote{The authors thank Sang-Hyun Kim for bringing this reference to our attention.}
Question \ref{motivation} above is especially salient when $g$ and $p$ are fixed, and we adopt this point of view in everything that follows.

There are few known obstructions to a graph $G$ having property $\cP_{g,p}$. 
The simplest is the presence of a clique of $G$ that is too large, as the size of a maximal complete subgraph of $\cC(S)$ is $3g-3+p$. 
A more subtle obstruction follows from a surprising fact proved by Bestvina, Bromberg, and Fujiwara: 
the graph $\cC(S)$ has finite chromatic number \cite{bbf,kim-koberda1}.

We introduce an invariant of a graph $G$ which we call the `nested complexity length' $\NCL(G)$ that controls the topological complexity of any surface whose curve graph contains $G$ as an induced subgraph (see \S\ref{definition} for a precise definition). 
The following is our main result, providing a new obstruction to $\cP_{g,p}$.
In fact, our calculation applies equally well to the \emph{clique graph} $\cC^{cl}(S)$ of $\cC(S)$, whose vertices are multicurves and with edges for disjointness.

\begin{thm}
\label{induced subgraph thm}
We have $\NCL(\cC(S)) = \NCL(\cC^{cl}(S)) = 6g-6+2p$.
\end{thm}

The nested complexity length of a graph is obtained via a supremum over all \emph{nested complexity sequences}, and while this definition is useful in light of the theorem above, we know of no non-exhaustive algorithm that computes the nested complexity length of a finite graph. Thus it is natural to ask:

\begin{question}
\label{NCL algorithm}
What is an algorithm to compute the nested complexity length of a finite graph? How can one find effective upper bounds?
\end{question}

As a starting point toward this question,  Proposition \ref{bp-upper-bound} gives an upper bound for the nested complexity length of a graph $G$ which is exponential in the maximal size of a complete bipartite subgraph of $G$.

We describe several corollaries of Theorem \ref{induced subgraph thm} below. 
The first concerns \emph{half-graphs}, a family of graphs that has attracted study in combinatorics and model theory. 
\begin{definition}
\label{half-graph}
Given an integer $n\geq 1$ and a graph $G$, we say that $G$ is a \emph{half-graph of height $n$} if there is a partition $\{a_1,\ldots,a_n\}\cup\{b_1,\ldots,b_n\}$ of the vertices of $G$ such that the edge $e(a_i,b_j)$ occurs if and only if $i\geq j$. 
The unique \emph{bipartite} half-graph of height $n$ is denoted $H_n$.
\end{definition}
In Example \ref{NCL example}, we observe that if $G$ is a half-graph of height $n$ then $\NCL(G)\geq n$. 
Because $\NCL$ is monotone on induced subgraphs, the following is an immediate consequence of Theorem \ref{induced subgraph thm}.

\begin{cor}
\label{omits half-graphs}
If $G$ is a half-graph of height $n$, with $n>6g-6+2p$, then $G$ does not have $\cP_{g,p}$.
\end{cor}

Since $H_n$ is 2-colorable and triangle-free, Corollary \ref{omits half-graphs} implies:

\begin{cor}
\label{obstructions1}
Chromatic number and maximal clique size are not sufficient to determine if a finite graph has $\cP_{g,p}$.
\end{cor}

In fact, a more dramatic illustration of Corollary \ref{obstructions1} can be made quantitatively.
By definition, $\cP_{g,p}$ is a \emph{hereditary graph property} (i.e. closed under isomorphism and induced subgraphs). 
Asymptotic enumeration of hereditary graph properties has been studied by many authors, resulting in a fairly precise description of possible ranges for growth rates \cite[Theorem 1]{bbw}. 
Combining Corollary \ref{omits half-graphs} with a result of Alon, Balogh, Bollob\'{a}s, and Morris \cite{abbm}, we obtain an upper bound on the asymptotic enumeration of $\cP_{g,p}$. 
The argument for the following is in \S\ref{definition}. Given $n>0$, let $\cP_{g,p}(n)$ denote the class of graphs with vertex set $[n]=\{1,\ldots,n\}$ satisfying $\cP_{g,p}$.

\begin{cor}\label{enumeration}
There is an $\epsilon>0$ such that, for large $n$, $|\cP_{g,p}(n)|\leq 2^{n^{2-\epsilon}}$.
\end{cor}

The set of graphs on $[n]$ satisfying the clique and chromatic tests for $\cP_{g,p}$ includes all $(3g+3-p)$-colorable graphs on $[n]$, and thus this set has size $2^{\Theta(n^2)}$.
In particular, the upper bound above cannot be obtained from the clique or chromatic number obstructions to $\cP_{g,p}$. 
This also strengthens the statement of Corollary \ref{obstructions1}: among the graphs on $[n]$ satisfying the clique and chromatic tests, the probability of possessing $\cP_{g,p}$ tends to $0$ as $n\to\infty$.

For \emph{monotone graph properties} (i.e.~closed under isomorphism and subgraphs), even more is known concerning asymptotic structure \cite{BBS}. However, we can use Corollary \ref{omits half-graphs} to show that in most cases, $\cP_{g,p}$ is not monotone.

\begin{cor}
\label{edge deletions}
If $3g+p\ge 6$ then $\cP_{g,p}$ is not a monotone graph property.
\end{cor}

\begin{proof}
If $3g+p\ge 6$ then $S$ contains a pair of disjoint incompressible subsurfaces that support essential nonperipheral simple closed curves. It follows that
the complete bipartite graph $K_{n,n}$ has property $\cP_{g,p}$ for all $n\in\mathbb{N}$. 
However, the half-graph $H_n$ is a subgraph of $K_{n,n}$.
\end{proof}

It is also worth observing that, for $3g+p\ge 6$, this result thwarts the possibility of using the Robertson-Seymour Graph Minor Theorem \cite{roberston-seymour} to characterize $\cP_{g,p}$ by a finite list of forbidden minors. 
Of course this would also require $\cP_{g,p}$ to be closed under edge contraction, which is already impossible just from the clique number restriction. 
Indeed, edge contraction of the graphs $K_{n,n}$ produces arbitrarily large complete graphs. 

\begin{remark}
\label{special cases 1}
Neither of the exceptional cases $\cP_{0,5}$ and $\cP_{1,2}$ are closed under edge contraction, as each contains a five-cycle but no four-cycles by Lemma \ref{multipartite obstructed}.
However, it remains unclear whether $\cP_{0,5}$ and $\cP_{1,2}$ are monotone graph properties; while Theorem \ref{induced subgraph thm} applies, complete bipartite graphs do not possess $\cP_{g,p}$ in these cases.
\end{remark}

Following Kim and Koberda, Question \ref{motivation} is closely related to the problem of which \emph{right-angled Artin groups} (RAAGs) embed in the mapping class group $\Mod(S)$ of $S$ \cite{koberda,kim-koberda1,kim-koberda2,kim-koberda3}. 
If the graph $G$ has $\cP_{g,p}$, then $A(G)$ is a RAAG subgroup of $\Mod(S)$ \cite[Theorem 1.1]{koberda}. 
The converse is false in general \cite[Theorem 3]{kim-koberda3}, but a related statement holds: if the RAAG $A(G)$ embeds as a subgroup of $\Mod(S)$, the graph $G$ is an induced subgraph of the clique graph $\cC^{cl}(S)$ \cite[Lemma 3.3]{kim-koberda1}.

By exploiting a construction by Erd\H os of graphs with arbitrarily large girth and chromatic number, Kim and Koberda produce `not very complicated' (precisely, cohomological dimension two) RAAGs that do not embed in $\Mod(S)$ \cite[Theorem 1.2]{kim-koberda1}.
Rephrasing Theorem \ref{induced subgraph thm} in this context, the graphs $H_n$ provide such examples which are `even less complicated'.

\begin{cor}
\label{RAAG subgroup restriction}
If $A(G)$ is a RAAG subgroup of $\Mod(S)$, then 
\[
\NCL(G)\le 6g-6+2p~.
\]
Moreover, for any $g$ and $p$ there exist bipartite graphs $G$ so that $A(G)$ does not embed in $\Mod(S)$.
\end{cor}

The nested complexity length of a graph $G$ is closely related to the \emph{centralizer dimension} of $A(G)$, i.e.~the longest chain of nontrivially nested centralizers in the group (this algebraic invariant is discussed elsewhere in the literature \cite{myasnikov-shumyatsky}).
In particular, it is straightforward from the definitions that the centralizer dimension of $A(G)$ is at least $\NCL(G)$.
The possible centralizers of an element of a RAAG have been classified by Servatius \cite{servatius}, and the characterization there would seem to suggest that equality does not hold in general.
Of course, this is impossible to check in the absence of a method to compute nested complesity length, and Question \ref{NCL algorithm} arises naturally. 

Our next corollary concerns the model theoretic behavior of $\cC(S)$. 
We focus on \emph{stability}, one of the most important and well-developed notions in modern model theory. 
Given a first-order structure, stability of its theory implies an
abstract notion of independence and dimension for definable sets in that
structure (see Pillay \cite{pillay} for details).
Stability can also be treated as a combinatorial property obtained from half-graphs. 
Given an integer $k\geq 1$, a graph $G$ is \emph{$k$-edge stable} if it does not contain any half-graph of height $n\geq k$ as an induced subgraph. 
We can thus rephrase Corollary \ref{omits half-graphs}.

\begin{cor}
$\cC(S)$ is $k$-edge stable for $k=6g-5+2p$.
\end{cor}

When considering the first-order theory of $\cC(S)$ in the language of graphs, this corollary implies that the edge formula (and thus any quantifier-free formula) is stable in the model theoretic sense \cite[Theorem 8.2.3]{te-zi}. 
Whether arbitrary formulas are stable remains an intriguing open question, and would likely require some understanding of quantifier elimination for the theory of $\cC(S)$ in some suitable expansion of the graph language. This aspect of the nature of curve graphs remains unexplored, and stability is only one among a host of natural questions about their first-order theories to pursue.

On the other hand, edge-stability of $\cC(S)$ alone has strong consequences for the structure of large finite subgraphs of $\cC(S)$, via Szemer\'{e}di's regularity lemma \cite{kom-sim,szemeredi}. 
In particular, Malliaris and Shelah show that if $G$ is a $k$-edge-stable graph, then the regularity lemma can be strengthened so that in Szemer\'{e}di partitions of large induced subgraphs of $G$, the bound on the number of pieces is significantly improved, there is no need for irregular pairs, and the density between each pair of pieces is within $\epsilon$ of $0$ or $1$~\cite{ma-sh}. 
Thus a consequence of our work is that the class of graphs with $\cP_{g,p}$ enjoys this stronger form of Szemer\'{e}di regularity.

Next we consider an explicit family of multipartite graphs.

\begin{definition}
\label{multipartite def}
Given integers $r,t>0$, let $K_r(t)$ denote the complete $r$-partite graph in which each piece of the partition has size $t$.
\end{definition}

In Example \ref{NCL example}, we show $\NCL(K_r(2))=2r$. Combined with the Erd\H{o}s-Stone Theorem, this implies a general relationship (Proposition \ref{density NCL prop}) between nested complexity length and the \emph{upper density} $\delta(G)$ of an infinite graph $G$, i.e. the supremum over real numbers $t$ such that $G$ contains arbitrarily large finite subgraphs of edge density $t$. In the case of the curve graph $\cC(S)$, we show in Lemma \ref{multipartite obstructed} that $K_r(t)$ is obstructed from having $\cP_{g,p}$ for large $r$. From this we obtain the following exact calculation of the upper density of the curve graph, which is proved in \S\ref{density}.

\begin{thm}
\label{density thm}
The upper density $\delta(\cC(S))$ is equal to $\displaystyle 1 - \frac{1}{g+\lfloor\frac{g+p}{2}\rfloor-1}$.
\end{thm}

\begin{remark}
\label{special cases 3}
The exceptional cases $(g,p)=(0,5)$ and $(1,2)$ are again remarkable in that they imply $\delta(\cC(S))=0$ in Theorem \ref{density thm}. Thus any family of graphs $\{G_n\}$ with $\cP_{g,p}$ and $|V(G_n)|\to \infty$ in these exceptional cases must satisfy $|E(G_n)| = o(|V(G_n)|^2).$
\end{remark}

\begin{question}
\label{density for exceptions}
Given $(g,p)=(0,5)$ or $(1,2)$, does there exist $\epsilon>0$ such that, for any family of graphs $\{G_n\}$ with $\cP_{g,p}$ and $|V(G_n)|\to\infty$ one has $|E(G_n)| = O(|V(G_n)|^{2-\epsilon})$?
\end{question}

Finally, we use the analysis of $\cP_{g,p}$ for $K_r(t)$ to extend Corollary \ref{obstructions1}.

\begin{cor}
\label{obstructions2}
Chromatic number, maximal clique size, and nested complexity length are not sufficient to determine if a finite graph has $\cP_{g,p}$.
\end{cor}

\begin{proof}
For $r=g+\lfloor\frac{g+p}{2}\rfloor$, the graph $K_r(2)$ does not have $\cP_{g,p}$ by Lemma \ref{multipartite obstructed}, but passes the clique, coloring, and NCL tests. Indeed the chromatic number of $\cC(S)$ is at least its clique number $3g-3+p$, which is greater than $r$ (the clique and chromatic number of $K_r(2)$). Moreover, 
by Example \ref{NCL example} and Theorem \ref{induced subgraph thm}, we have $\NCL(K_r(2))=2r \le \NCL(\cC(S))$.
\end{proof}

\section*{acknowledgements}

The authors thank Tarik Aougab, Ian Biringer, Josh Greene, Sang-Hyun Kim, and Thomas Koberda for helpful comments and discussions.


\section{Notation and conventions}
\label{sec: notation}

We briefly list definitions and notation relevant in this paper, with the notable exception of nested complexity length (found in \S\ref{definition}). For background and context see Farb and Margalit \cite{farb-margalit}.

A \emph{graph} $G$ consists of a set of vertices $V(G)$ and an edge set $E(G)$ which is a subset of unordered distinct pairs from $V(G)$; we denote the edge between vertices $v$ and $w$ by $e(v,w)\in E(G)$.
A subgraph $H \subset G$ is \emph{induced} if $v,w\in H$ and $e(v,w)\in E(G)$ implies that $e(v,w)\in E(H)$.
The \emph{closed neighborhood} of a vertex $v\in V(G)$ is the set of vertices 
\[N[v]=\{v\} \cup \{u\in V(G):e(u,v)\in E(G)\}.
\]

Given a graph $G$ the right-angled Artin group corresponding to $G$, denoted $A(G)$, is defined by the following group presentation: 
the generators of $A(G)$ are given by $V(G)$ and there is a commutation relation $vw=wv$ for every edge $e(v,w)\in E(G)$. 

Recall that we are concerned with the hyperbolizable surface $S$ of genus $g$ with $p$ punctures, so we assume that $2g+p>2$.
The \emph{mapping class group} of $S$, denoted $\Mod(S)$, is the group $\pi_0(\mathrm{Homeo}^+(S))$.

By a \emph{curve} we mean the isotopy class of an essential nonperipheral embedded loop on $S$, and we refer to the union of curves which can be made simultaneously disjoint as a \emph{multicurve}.
The \emph{curve graph} $\cC(S)$ is the graph consisting of a vertex for each curve on $S$, and so that a pair of curves span an edge when the curves can be realized disjointly.
The \emph{clique graph} $\cC^{cl}(S)$ of the curve graph is the graph obtained as follows: 
The vertices of $\cC^{cl}(S)$ correspond to cliques of $\cC(S)$ (i.e.~multicurves), and two cliques are joined by an edge when they are simultaneously contained in a maximal clique (i.e~can be realized disjointly). 
The curve graph is the subgraph of the clique graph induced by the one-cliques.

We strengthen the assumption on $g$ and $p$ above to $3g+p\ge 5$. When $(g,p)=(0,3)$ the curve graph as defined above has no vertices.
When $(g,p)=(0,4)$ or $(1,1)$, the curve graph has no edges, and the matter of deciding if a graph has $\cP_{g,p}$ in these cases is trivial. 
The common alteration of the definition of these curve graphs yields the \emph{Farey graph}. 
We make no comment on induced subgraphs of the Farey graph, though a comprehensive classification can be made.

Whenever we refer to a \emph{subsurface} $\Sigma \subset S$ we make the standing assumption that $\Sigma$ is a disjoint union of closed incompressible homotopically distinct subsurfaces with boundary of $S$, i.e. the inclusion $\Sigma \subset S$ induces an injection on the fundamental groups of components of $\Sigma$. 
We write $[\Sigma]$ for the isotopy class of $\Sigma$.

Given a subsurface $\Sigma \subset S$, we say that a curve $\gamma$ is \emph{supported on} $\Sigma$ if it has a representative which is either contained in an annular component of $\Sigma$, or is a nonperipheral curve in a (necessarily non-annular) component. 
The curve $\gamma$ is \emph{disjoint from} $\Sigma$ if it is has a representative disjoint from $\Sigma$.
We say that $\gamma$ is \emph{transverse} to $\Sigma$, written $\gamma \pitchfork \Sigma$, if it is not supported on $\Sigma$ and not disjoint from $\Sigma$.
A multicurve $\gamma$ is supported on (resp.~disjoint from) $\Sigma$ if each of its components is supported on (resp.~disjoint from) $\Sigma$, and $\gamma$ is transverse to $\Sigma$ if at least one of its components is.
Each of these above definitions applies directly for curves and isotopy classes of subsurfaces.

We note that the definitions above may be nonstandard, as they are made with our specific application in mind. 
For example, in our terminology the core of an annular component of $\Sigma$ is both supported on and disjoint from the subsurface $\Sigma$.

Given a pair of subsurfaces $\Sigma_1$ and $\Sigma_2$, we write $[\Sigma_1] \le [\Sigma_2]$ when every curve supported on $\Sigma_1$ is supported on $\Sigma_2$. 
If $[\Sigma_1] \le [\Sigma_2]$ and $[\Sigma_1] \ne [\Sigma_2]$ (denoted $[\Sigma_1]<[\Sigma_2]$), we say that $[\Sigma_1]$ and $[\Sigma_2]$ are \emph{nontrivially nested}.

Given a collection of curves $C$, we let $[\Sigma(C)]$ indicate the isotopy class of the minimal subsurface of $S$, with respect to the partial ordering just defined, that supports the curves in $C$. 
Concretely, a representative of $[\Sigma(C)]$ can be obtained by taking the
union of regular neighborhoods of the curves in $C$ (for suitably small
regular neighborhoods that depend on the realizations of the curves in $C$) and filling in contractible complementary components of the result with disks.

A set of curves $C$ supported on a subsurface $\Sigma$ \emph{fills} the subsurface if $[\Sigma(C)]=[\Sigma]$.
Concretely, $C$ fills $\Sigma$ if the core of each annular component of $\Sigma$ is in $C$, and if the complement in $\Sigma$ of a realization of the remaining curves consists of peripheral annuli and disks.

Given a subsurface $\Sigma$, we write $\xi(\Sigma)$ for the number of components of a maximal multicurve supported on $\Sigma$.


\section{Topological complexity of subsurfaces}
\label{lemmas}

Since the ambient surface $S$ is fixed throughout, it should not be surprising that any nontrivially nested chain of subsurfaces of $S$ has bounded length. 
We make this explicit below, using $\xi(S)$ to keep track of `how much' of the surface has been captured by subsurfaces from the chain.
In fact, this is not quite enough to notice nontrivial nesting of subsurfaces, as annuli can introduce complications. 
We keep track of this carefully in Lemma \ref{chain lemma} below.
Throughout this section, $\Sigma_1$ and $\Sigma_2$ refer to a pair of subsurfaces of $S$.

\begin{lemma}
\label{inclusion lemma}
If $[\Sigma_1] \le [\Sigma_2]$, then $\Sigma_1$ is isotopic to a subsurface of $\Sigma_2$.
\end{lemma}

The reader is cautioned that the converse is false: consider a one-holed torus subsurface $T$ and let $T'$ indicate the disjoint union of $T$ with an annulus isotopic to $\partial T$. 
Though $T'$ is isotopic to a subsurface of $T$, the curve $\partial T$ is supported on $T'$ but not on $T$. Thus $[T'] \not \le [T]$.

\begin{proof}
Choose a collection of curves $C$ filling $\Sigma_1$. 
Since $[\Sigma_1] \le [\Sigma_2]$, each curve in $C$ is supported on $\Sigma_2$. 
It follows that there are small enough regular neighborhoods of representatives of the curves in $C$ that are contained in $\Sigma_2$.
Filling in contractible components of the complement produces a subsurface isotopic to $\Sigma_1$ inside $\Sigma_2$.
\end{proof} 

\begin{lemma}
\label{equality lemma}
If $[\Sigma_1] \le [\Sigma_2]$ and $[\Sigma_2] \le [\Sigma_1]$, then $[\Sigma_1]=[\Sigma_2]$.
\end{lemma}

\begin{proof}
Fix representatives $\Sigma_1 \subset \Sigma_2$, using the previous lemma. 
Let $C$ be a collection of curves that fill $\Sigma_2$. 
Since $[\Sigma_2]\le [\Sigma_1]$ the curves in $C$ can be realized on $\Sigma_1$ so that they have regular neighborhoods contained in $\Sigma_1$. 
Their complement in $\Sigma_2$ is a collection of disks and peripheral annuli, so their complement in $\Sigma_1$ must also be a collection of disks and peripheral annuli. 
Hence $C$ fills $\Sigma_1$ and $[\Sigma_1] = [\Sigma(C)] = [\Sigma_2]$.
\end{proof}

\begin{lemma}
\label{strict inequality}
Suppose that $[\Sigma_1] < [\Sigma_2]$. Then one of the following holds:
\begin{enumerate}
\item
we have $\xi(\Sigma_1) < \xi(\Sigma_2)$, or
\item 
there exists a curve $\alpha$ which is the core of an annular component of $\Sigma_1$ but not disjoint from $\Sigma_2$.
\end{enumerate}
\end{lemma}

\begin{proof}
By Lemma \ref{equality lemma} we have $[\Sigma_2] \not \le [\Sigma_1]$. Thus there is a curve $\gamma$ supported on $\Sigma_2$ that is not supported on $\Sigma_1$.
If $\gamma$ were the core of an annular component of $\Sigma_2$, then $\xi(\Sigma_1)<\xi(\Sigma_2)$ would be immediate.
Likewise if $\gamma$ were disjoint from $\Sigma_1$ then $\xi(\Sigma_1)< \xi(\Sigma_2)$ again.

We are left with the case that $\gamma$ is transverse to $\Sigma_1$.
Choose a boundary component $\alpha$ of $\Sigma_1$ intersected essentially by $\gamma$. 
Evidently, this curve must be supported on $\Sigma_2$. 
Since $\alpha$ is supported on $\Sigma_2$, either one has $\xi(\Sigma_1)<\xi(\Sigma_2)$ or $\alpha$ is supported as well on $\Sigma_1$. 
In this case, $\alpha$ is the core of an annular component of $\Sigma_1$ intersected essentially by $\gamma$. 
If $\alpha$ could be made disjoint from $\Sigma_2$ then its intersection with $\gamma$ would be inessential, so we are done.
\end{proof}

\begin{lemma}
\label{chain lemma}
Suppose that $\emptyset\ne[\Sigma_1] < [\Sigma_2] < \ldots < [\Sigma_n]$ is a nontrivially nested chain of subsurfaces of $S$. Then $n \le 6g-6+2p$.
\end{lemma}

\begin{proof}
Choose $k$ with $1\le k \le n-1$, and let $c_k = \xi(\Sigma_k)$. 
We construct, inductively on $k$, a (possibly empty) multicurve $\omega$. 
Evidently, we have $[\Sigma_1] \le [\Sigma_2]$, so that by Lemma \ref{strict inequality} either $c_1 < c_2$, or there is a curve $\alpha_1$ isotopic to the core of an annular component of $\Sigma_1$ but not disjoint from $\Sigma_2$.
In the second case, add $\alpha_1$ to $\omega$.
Note that in the latter case there is a representative for $\alpha_1$ which is contained in $\Sigma_1$.

We continue inductively.
Since we have $[\Sigma_k] < [\Sigma_{k+1}]$, Lemma \ref{strict inequality} guarantees that either $c_k < c_{k+1}$, or there is a curve $\alpha_k$ isotopic to the core on an annular component of $\Sigma_k$ but not disjoint from $\Sigma_{k+1}$.
Suppose that, for some $i<k$, the curve $\alpha_i$ is another component of $\omega$ supported on $\Sigma_i$.
Since $\alpha_k$ has a representative disjoint from $\Sigma_k$, and $[\Sigma_i] \le [\Sigma_k]$, $\alpha_k$ has a representative disjoint from $\Sigma_i$. 
Moreover, because $\alpha_i$ cannot be made disjoint from $\Sigma_{i+1}$ but $\alpha_k$ can be (since $[\Sigma_{i+1}] \le [\Sigma_k]$), the curves $\alpha_i$ and $\alpha_k$ are not isotopic.
It follows that the curve $\alpha_k$ may be added to the multicurve $\omega$ so that $\omega$ remains a collection of disjoint curves, and so that its number of components increases by one.

At each step of the chain $[\Sigma_1] < [\Sigma_2] < \ldots < [\Sigma_n]$, either $c_k$ strictly increases, or $\omega$ gains another component.
Since $[\Sigma_1]\ne\emptyset$ and $[\Sigma_n] \le [S]$, we have $1\le c_1$ and $c_n \le \xi(S)$.
The number of components of $\omega$ is also at most $\xi(S)$, so we conclude that
\[
n \le 2\xi(S) = 6g-6+2p.\qedhere
\] \end{proof}

Finally, we will make use of a straightforward certificate that $[\Sigma_1] < [\Sigma_2]$.

\begin{lemma}
\label{certificate}
Suppose that $[\Sigma_1] \le [\Sigma_2]$. If $\alpha$ is a curve disjoint from $\Sigma_1$, $\beta$ is a curve supported on $\Sigma_2$, and $\alpha$ and $\beta$ intersect essentially, then $[\Sigma_1] < [\Sigma_2]$.
\end{lemma}

\begin{proof}
Because $\alpha$ is disjoint from $\Sigma_1$, it has a representative disjoint form $\Sigma_1$. If $\beta$ were supported on $\Sigma_1$, it would have a representative contained in $\Sigma_1$, contradicting the assumption that $\alpha$ and $\beta$ intersect essentially. Thus $\beta$ is a curve supported on $\Sigma_2$ but not supported on $\Sigma_1$.
\end{proof}


\section{The nested complexity length of a graph}
\label{definition}

The topological hypotheses of Lemma \ref{certificate} suggest a natural combinatorial parallel, which we capture in the definition of nested complexity length.

\begin{definition}
\label{nes com def}
Let $G$ be a graph.
\begin{enumerate}[(1)]
\item Given $b_1,\ldots,b_n\in V(G)$, we say that $(b_1,\ldots,b_n)$ is a \emph{nested complexity sequence} for $G$ if for each $1\le k\le n-1$ there is $a_k\in V(G)$ such that $b_1,\ldots,b_k\in N[a_k]\not\ni b_{k+1}$.
\item The \emph{nested complexity length of $G$}, denoted $\NCL(G)$, is given by
\[
\NCL(G):=\sup \left\{ n \ | \ (b_1,\ldots,b_n) \text{ is a nested complexity sequence for } G\right \}.
\]
\end{enumerate}
\end{definition}

\begin{figure}[H]
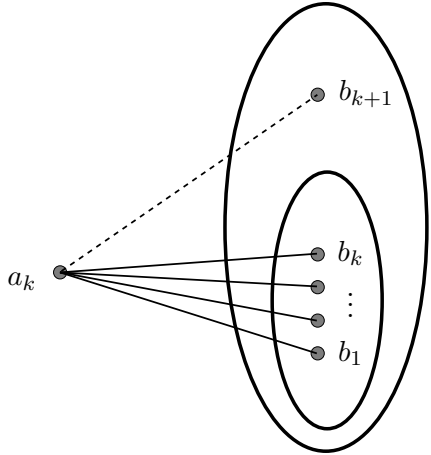

	\begin{lpic}[clean]{nestedComplexity(,6cm)}
		\lbl[]{-10,55;$a_k$}
		\lbl[]{100,115;$b_{k+1}$}
		\lbl[]{95,64;$b_k$}
		\lbl[]{95,50;$\vdots$}
		\lbl[]{95,32;$b_1$}
	\end{lpic}
	\caption{A graph acquires some nested complexity. Dotted lines indicate edges that are necessarily absent.}
	\label{nested complexity}
\end{figure}

Note that in the definition of a nested complexity sequence, $a_k$ may be equal to $b_i$, for $1\le i \le k$; see Figure \ref{nested complexity} for a schematic in which this is not the case. To highlight this subtlety, and as a first step toward an answer to Question \ref{NCL algorithm}, we prove an upper bound for $\NCL(G)$ in terms of a maximal complete bipartite subgraph of $G$.

\begin{prop}\label{bp-upper-bound}
Let $G$ be a graph. Fix $m,n>0$ and suppose $G$ does not contain a subgraph isomorphic to $K_{m,n}$. Then $\NCL(G)\le 2^{m+n+1}-2$. 
\end{prop}
\begin{proof}
Given $k>0$, let $s_k=2^{k+1}-2$. Note that $s_k=2(s_{k-1}+1)$. Set $N=2^{m+n+1}-1=s_{m+n}+1$. For a contradiction, suppose $(b_1,\ldots,b_N)$ is a nested complexity sequence for $G$, witnessed by $a_1,\ldots,a_{N-1}$. 

We inductively produce values $j_1<\ldots<j_{m+n}< N$ below, such that for all $1\le r\le m+n$, $j_r\le s_r$ and $a_{j_r}\neq b_i$ for all $i\le j_r$. 
With these indices chosen, the set $\{b_{j_1},\ldots,b_{j_m}\}\cup\{a_{j_{m+1}}\ldots,a_{j_{m+n}}\}\subseteq V(G)$ produces a (not necessarily induced) subgraph of $G$ isomorphic to $K_{m,n}$, a contradiction.

If $a_1\neq b_1$ then we set $j_1=1$. Otherwise, if $a_1=b_1$ then $b_1,b_2$ are independent vertices in $N[a_2]$. So $a_2\not\in \{b_1,b_2\}$, and we set $j_1=2$. 
 
Fix $1\leq r<m+n$, and suppose we have defined $j_r$ as above. If $a_{j_r+1}\neq b_i$ for all $i\le j_r+1$ then we let $j_{r+1}=j_r+1$. Otherwise, let $k=j_r+1$ and suppose $a_k=b_{i_k}$ for some $i_k\le k$. We will find $t$ such that $k<t\le 2k$ and $a_t\neq b_i$ for any $i\le t$. By induction, $2k\le 2(s_r+1)=s_{r+1}<N$, and so setting $j_{r+1}=t$ finishes the inductive step of the construction.
 
Suppose no such $t$ exists. Then for all $k\le t\le 2k$, we have $a_t=b_{i_t}$ for some $i_t\le t$. Fix $1\le s\le k$. For any $0\le u< s$, we have $a_{k+u}=b_{j_{k+u}}$, and so $b_{k+u+1}\not\in N[b_{j_{k+u}}]$. Therefore, for any $0\le u< s$, $b_{k+u+1}$ and $b_{j_{k+u}}$ are independent vertices in $N[a_{k+s}]$, which means $a_{k+s}\not\in\{b_{k+u+1},b_{j_{k+u}}\}$. In other words, we have shown that for all $1\le s\le k$,
 \[
 j_{k+s}\not\in\{k+1,\ldots,k+s,j_k,j_{k+1},\ldots,j_{k+s-1}\}.
 \]
It follows that $j_k,j_{k+1},\ldots,j_{2k}$ are $k+1$ distinct elements of $\{1,\ldots,k\}$.
\end{proof}

We make note of two useful properties of $\NCL$ that follow immediately from the definitions.

\begin{prop}$~$
\label{prop NCL}
\begin{enumerate}
\item If $H$ is an induced subgraph of $G$, then $\NCL(H)\le \NCL(G)$.
\label{NCL monotone}
\item We have $\NCL(G) \le |V(G)|$.
\label{NCL bound}
\end{enumerate}
\end{prop}

We also give the following examples, which are heavily exploited in the results outlined in \S\ref{intro}.

\begin{example}\label{NCL example}$~$
\begin{enumerate}
\item Let $G$ be a half-graph of height $n$. 
We may partition the vertices as $V(G)=\{a_1,\ldots,a_n\}\cup\{b_1,\ldots,b_n\}$, where $e(a_i,b_j)\in E(G)$ if and only if $i\ge j$. 
Then $(b_1,\ldots,b_n)$ is a nested complexity sequence for $G$, witnessed by $a_1,\ldots,a_{n-1}$. 
Therefore $\NCL(G)\ge n$.
\item Let $G$ be the multipartite graph $K_r(2)$. Let $V(G)=\{b_1,\ldots,b_{2r}\}$ where $e(b_{2k-1},b_{2k})\not\in E(G)$ for $1\leq k\leq r$. 
Set $a_1=b_1$. 
For $2\leq k \le r$ let $a_{2k-2}=b_{2k}$ and $a_{2k-1}=b_{2k-1}$.  
Then $(b_1,\ldots,b_{2r})$ is a nested complexity sequence for $G$, witnessed by $a_1,\ldots,a_{2r-1}$. 
Combined with Proposition \ref{prop NCL}(2), we have $\NCL(G)=2r$.
\end{enumerate}
\end{example}

With the first example in hand, we immediately derive Corollary \ref{omits half-graphs} from Theorem \ref{induced subgraph thm}. 
We can also give the proof of Corollary \ref{enumeration}.

\begin{proof}[Proof of Corollary \ref{enumeration}]
Given $k>0$, let $\mathcal{U}(k)$ denote the class of graphs $G$ for which there is a partition $V(G)=\{a_i:1\le i\le k\}\cup\{b_J:J\subseteq\{1,\ldots,k\}\}$ such that $e(a_i,b_J)$ holds if and only if $i\in J$. 
Any graph in $\mathcal{U}(k)$ contains an induced half-graph of height $k$ (take $\{a_i:1\le i\le k\}\cup\{b_{J_i}:1\le i\le k\}$, where $J_i=\{i,\ldots,k\}$). 
By Corollary \ref{omits half-graphs}, every graph with $\cP_{g,p}$ omits the class $\mathcal{U}(k)$ for $k>6g-6+2p$. 
The result now follows immediately from Theorem 2 of Alon, et al.~\cite{abbm}.
\end{proof}


\section{Proof of Theorem \ref{induced subgraph thm}}
\label{thm proof}

Recall that $[\Sigma(C)]$ refers to the minimal isotopy class, with respect to inclusion, of a subsurface of $S$ that supports the curves in $C$.

\begin{proof}[Proof of Theorem \ref{induced subgraph thm}]
As $\cC(S)$ is an induced subgraph of $\cC^{cl}(S)$, Proposition \ref{prop NCL}(1) ensures that $\NCL(\cC(S)) \le \NCL(\cC^{cl}(S))$. We proceed by showing that $6g-6+2p$ is simultaneously a lower bound for $\NCL(\cC(S))$ and an upper bound for $\NCL(\cC^{cl}(S))$.

For the lower bound, choose a maximal multicurve $\{\gamma_1,\ldots,\gamma_{3g-3+p}\}$. 
For each curve $\gamma_i$, choose a \emph{transversal} $\eta_i$, i.e.~a curve intersecting $\gamma_i$ essentially and disjoint from $\gamma_j$ for $j\ne i$.
That such collections of curves exist is routine (e.g.~a `complete clean
marking' in the language of Masur and Minsky \cite[\S2.5]{masur-minsky2} is an even more restrictive example, see Figure \ref{marking}).

\begin{figure}[H]
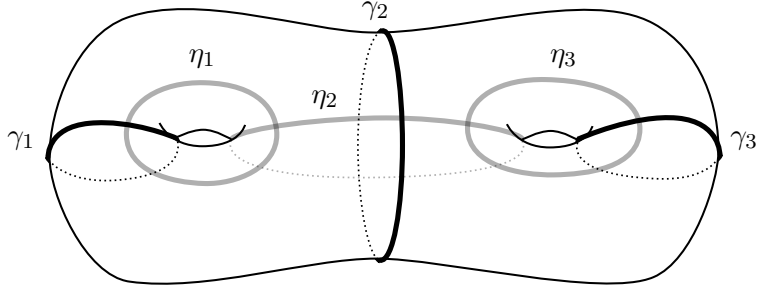

	\centering
	\begin{lpic}{pantsDecomp(9cm)}
		\lbl[]{-10,60;$\gamma_1$}
		\lbl[]{137,114;$\gamma_2$}
		\lbl[]{290,60;$\gamma_3$}
		\lbl[]{65,94;$\eta_1$}
		\lbl[]{116,77;$\eta_2$}
		\lbl[]{215,94;$\eta_3$}
	\end{lpic}
	\caption{A bold maximal multicurve $\{\gamma_1,\gamma_2,\gamma_3\}$ and a lighter set of transversals $\{\eta_1,\eta_2,\eta_3\}$.}
	\label{marking}
\end{figure}

For each $1\le i \le 6g-6+2p$, let the curves $\alpha_i$ and $\beta_i$ be given by
\begin{align*}
\beta_i & = 
	\begin{cases}
		\gamma_i  & 1 \le i \le 3g-3+p \\
		\eta_{i-(3g-3+p)}  & 3g-3+p < i \le 6g-6+2p~,
	\end{cases} \\
\alpha_i & =
	\begin{cases}
		\eta_{i+1} & 1 \le i \le 3g-4+p \\
		\gamma_{i-(3g-4+p)} & 3g-4+p < i \le 6g-7+2p~.
	\end{cases}
\end{align*}
It is straightforward to check that $(\beta_1,\ldots,\beta_{6g-6+2p})$ forms a nested complexity chain for $\cC(S)$, with witnessing curves $(\alpha_1,\ldots,\alpha_{6g-7+2p})$.

Towards the upper bound, suppose $(\beta_1,\ldots,\beta_n)$ is a nested complexity sequence for $\cC^{cl}(S)$,
 so that there exists a vertex $\alpha_k$ in $\cC^{cl}(S)$ with $\beta_1,\ldots,\beta_k\in N[\alpha_k]\not\ni\beta_{k+1}$ for each $1\le k\le n-1$ as in Definition \ref{nes com def}. 
For $1\leq k\leq n$, let $B_k=\{\beta_1,\ldots,\beta_k\}$, and let $[\Sigma_k]=[\Sigma(B_k)]$. 

Choose $k \le n-1$.
Because $B_k \subset B_{k+1}$, we have $[\Sigma_k] \le [\Sigma_{k+1}]$; as $B_k \subset N[\alpha]$, $\alpha_k$ is disjoint from $\Sigma_k$; and since $\beta_{k+1} \in B_{k+1}$ we have that $\beta_{k+1}$ is supported on $\Sigma_{k+1}$. 
Because $\beta_{k+1} \not\in N[\alpha_k]$, $\alpha_k$ and $\beta_{k+1}$ are independent vertices of $\cC^{cl}(S)$, and so there are components of the multicurves $\alpha_k$ and $\beta_{k+1}$ that intersect essentially. 
Lemma \ref{certificate} now applies, so that $[\Sigma_k] < [\Sigma_{k+1}]$.
Thus $\emptyset \ne [\Sigma_1] < [\Sigma_2] < \ldots < [\Sigma_n]$ is a nontrivially nested chain of subsurfaces, and Lemma \ref{chain lemma} implies that $n\le 6g-6+2p$.
\end{proof}

\begin{remark}
\label{tighter bound NCL}
The upper bound can be strengthened for $\cC(S)$ under the additional assumption that the antigraphs of the subgraphs induced on the vertices $\{b_1,\ldots,b_k\}$ are connected for each $k$ (indeed, in this case the negative Euler characteristic of $\Sigma_k$ is strictly less than that of $\Sigma_{k+1}$). 
In particular, if $n\ge 2g+p$ then the bipartite half-graph $H_n$ does not have $\cP_{g,p}$.
\end{remark}


\section{Obstructing $K_r(t)$ and the upper density of the curve graph}
\label{density}

We turn to $K_r(t)$ and the upper density of curve graphs.

\begin{definition}
{Let $G$ be a graph.
\begin{enumerate}
\item If $|G|=n>1$ the \emph{density of $G$} is
\[
\delta(G)=\frac{|E(G)|}{{n\choose 2}}~.
\]
\item If $G$ is infinite the \emph{upper density of $G$} is
\[
\delta(G) = \limsup_{H\subseteq G,~|H|\rightarrow\infty}\delta(H)~.
\]
\end{enumerate}}
\end{definition}

In other words, given an infinite graph $G$, $\delta(G)$ is the supremum over all $\alpha\in(0,1]$ such that $G$ contains arbitrarily large finite subgraphs of density at least $
\alpha$.
Given a fixed $r>0$, it is easy to verify $\lim_{t\rightarrow\infty}\delta(K_r(t))=1-\frac{1}{r}$. 
It is a consequence of a remarkable theorem of Erd\H{o}s and Stone that the graphs $K_r(t)$ witness the densities of all infinite graphs. 
We record this precisely in language most relevant for our application (\cite[Ch.~IV]{bollobas}, \cite{erdos-stone}).

\begin{namedtheorem}[Erd\H os-Stone]\label{erdos stone cor}
Fix $r>0$. For any infinite graph $G$, if $\delta(G)>1-\frac{1}{r}$ then $K_{r+1}(t)$ is a subgraph of $G$ for arbitrary $t$.
\end{namedtheorem} 

Using this theorem, we obtain the following relationship between density and nested complexity length. 

\begin{prop}\label{density NCL prop}
Let $G$ be an infinite graph. If $\delta(G)<1$ then 
\[
\NCL(G)\ge\frac{2}{1-\delta(G)}.
\]
\end{prop}
\begin{proof}
If this inequality fails then we have $\delta(G)>1-\frac{2}{\NCL(G)}$. By the Erd\H{o}s-Stone Theorem, there is $r>\frac{1}{2}\NCL(G)$ such that $K_r(t)$ is a subgraph of $G$ for arbitrarily large $t$. Let $w$ be the size of the largest finite clique in $G$, which exists since $\delta(G)<1$. If we consider a copy of $K_r(w+1)$ in $G$, it follows that each piece of the partition contains a pair of independent vertices. Therefore $K_r(2)$ is an \emph{induced} subgraph of $G$. By Proposition \ref{prop NCL} and Example \ref{NCL example}, $\NCL(G)\ge 2r$, which contradicts the choice of $r$.
\end{proof}

Note that if $G$ is complete then $\NCL(G)=1$. Therefore an infinite graph with upper density $1$ need not have large nested complexity length.

Finally, we obstruct the multipartite graphs $K_r(t)$ from having $\cP_{g,p}$ for large enough $r$ and $t\ge 2$, and employ this fact in the proof of Theorem \ref{density thm}.

\begin{lemma}
\label{max subsurfaces calculation}
The maximum number of pairwise disjoint subsurfaces of $S$ which are not annuli or pairs of pants is $g+\lfloor \frac{g+p}{2} \rfloor -1$.
\end{lemma}

\begin{proof}
It suffices to consider a pairwise disjoint sequence of subsurfaces in which each component is a four-holed sphere or a one-holed torus; any more complex subsurface can be cut further without decreasing the number of non-annular and non-pair of pants subsurfaces.
Suppose there are $T$ one-holed tori and $F$ four-holed spheres. 
The dimension of the homology of $S$ requires $0\leq T \leq g$. 
Additivity of the Euler characteristic under disjoint union implies that $0\le T+2F \le 2g+p-2$.
Maximizing $T+F$ on this polygon is routine, and the solution is $T=g$ and $F=\lfloor \frac{g+p}{2} \rfloor -1$. 
Moreover, it is straightforward to construct such a collection of subsurfaces of $S$.
\end{proof}

This is enough to obstruct $K_r(t)$ for large $r$.

\begin{lemma}
\label{multipartite obstructed}
For $t>1$, $K_r(t)$ has $\cP_{g,p}$ if and only if $r \le g+\lfloor\frac{g+p}{2}\rfloor-1$.
\end{lemma}

\begin{proof}
Let $\ell = g+\lfloor\frac{g+p}{2}\rfloor-1$.
First, suppose $\Sigma_1,\ldots,\Sigma_\ell$ is a sequence of pairwise disjoint subsurfaces consisting of tori and four-holed spheres, as guaranteed by Lemma \ref{max subsurfaces calculation}. 
The curves supported on the $\Sigma_i$ induce $K_\ell(\infty)$ as a subgraph of $\cC(S)$. 
Hence $K_r(t)$ is an induced subgraph of $\cC(S)$ for all $t$ and $r\leq \ell$. 
Conversely, suppose towards contradiction that $K_{\ell+1}(2)$ is an induced subgraph of $\cC(S)$, and let $V_1,\ldots,V_{\ell+1}$ be the partition of its vertices.
For each $i\ne j$, both curves in $V_i$ are disjoint from the curves in $V_j$,
so the subsurfaces $\Sigma(V_i)$ and $\Sigma(V_j)$ are disjoint.
Moreover, since the curves in $V_i$ intersect, $\Sigma(V_i)$ is a connected surface that is not an annulus or a pair of pants.
Thus $\Sigma(V_1),\ldots,\Sigma(V_{\ell+1})$ is a sequence of disjoint non-annular and non-pair of pants subsurfaces, contradicting Lemma \ref{max subsurfaces calculation}.
\end{proof}

We can now prove Theorem \ref{density thm}.

\begin{proof}[Proof of Theorem \ref{density thm}]
Let $\ell = g+\lfloor\frac{g+p}{2}\rfloor-1$ again.
By Lemma \ref{multipartite obstructed}, $K_\ell(t)$ has $\cP_{g,p}$ for all $t$, so that $\delta(\cC(S))\geq 1-\frac{1}{\ell}$.
If $\delta(\cC(S)) > 1-\frac{1}{\ell}$ then, as in the proof of Proposition \ref{density NCL prop}, the Erd\H{o}s-Stone Theorem and finite clique number of $\cC(S)$ together imply that $K_{\ell+1}(2)$ is an \emph{induced} subgraph of $\cC(S)$, violating Lemma \ref{multipartite obstructed}.
\end{proof}


\begin{thebibliography}{ABBM}

\bibitem[ABBM]{abbm}
Noga Alon, J{\'o}zsef Balogh, B{\'e}la Bollob{\'a}s, and Robert Morris.
\newblock {The structure of almost all graphs in a hereditary property}.
\newblock {\em J. Combin. Theory Ser. B} {\bf 101}(2011), 85--110.

\bibitem[BBS]{BBS}
J{\'o}zsef Balogh, B{\'e}la Bollob{\'a}s, and Mikl{\'o}s Simonovits.
\newblock {The number of graphs without forbidden subgraphs}.
\newblock {\em J. Combin. Theory Ser. B} {\bf 91}(2004), 1--24.

\bibitem[BBW]{bbw}
J{\'o}zsef Balogh, B{\'e}la Bollob{\'a}s, and David Weinreich.
\newblock {The penultimate rate of growth for graph properties}.
\newblock {\em European J. Combin.} {\bf 22}(2001), 277--289.

\bibitem[BBF]{bbf}
Mladen Bestvina, Ken Bromberg, and Koji Fujiwara.
\newblock {Constructing group actions on quasi-trees and applications to
  mapping class groups}.
\newblock {\em Publ. Math. Inst. Hautes {\'E}tudes Sci.} {\bf 122}(2015),
  1--64.


\bibitem[Bol]{bollobas}
B{{\'e}}la Bollob{{\'a}}s.
\newblock {\em Extremal graph theory}, volume~11 of {\em London Mathematical
  Society Monographs}.
\newblock Academic Press, Inc. [Harcourt Brace Jovanovich, Publishers],
  London-New York, 1978.



\bibitem[CW1]{crisp-wiest1}
John Crisp and Bert Wiest.
\newblock {Embeddings of graph braid and surface groups in right-angled {A}rtin
  groups and braid groups}.
\newblock {\em Algebr. Geom. Topol.} {\bf 4}(2004), 439--472.

\bibitem[CW2]{crisp-wiest2}
John Crisp and Bert Wiest.
\newblock {Quasi-isometrically embedded subgroups of braid and diffeomorphism
  groups}.
\newblock {\em Trans. Amer. Math. Soc.} {\bf 359}(2007), 5485--5503.

\bibitem[EET]{ehrlich-even-tarjan}
G.~Ehrlich, S.~Even, and R.~E. Tarjan.
\newblock {Intersection graphs of curves in the plane}.
\newblock {\em J. Combinatorial Theory Ser. B} {\bf 21}(1976), 8--20.

\bibitem[ES]{erdos-stone}
P.~Erd{{\"o}}s and A.~H. Stone.
\newblock {On the structure of linear graphs}.
\newblock {\em Bull. Amer. Math. Soc.} {\bf 52}(1946), 1087--1091.

\bibitem[FM]{farb-margalit}
Benson Farb and Dan Margalit.
\newblock {\em A primer on mapping class groups}, volume~49 of {\em Princeton
  Mathematical Series}.
\newblock Princeton University Press, Princeton, NJ, 2012.

\bibitem[KK1]{kim-koberda1}
Sang-Hyun Kim and Thomas Koberda.
\newblock {An obstruction to embedding right-angled {A}rtin groups in mapping
  class groups}.
\newblock {\em Int. Math. Res. Not. IMRN} (2014), 3912--3918.

\bibitem[KK2]{kim-koberda2}
Sang-Hyun Kim and Thomas Koberda.
\newblock {Anti-trees and right-angled {A}rtin subgroups of braid groups}.
\newblock {\em Geom. Topol.} {\bf 19}(2015), 3289--3306.

\bibitem[KK3]{kim-koberda3}
Sang-Hyun Kim and Thomas Koberda.
\newblock {Right-angled {A}rtin groups and finite subgraphs of curve graphs}.
\newblock {\em Osaka J. Math.} {\bf 53}(2016), 705--716.

\bibitem[Kob]{koberda}
Thomas Koberda.
\newblock {Right-angled {A}rtin groups and a generalized isomorphism problem
  for finitely generated subgroups of mapping class groups}.
\newblock {\em Geom. Funct. Anal.} {\bf 22}(2012), 1541--1590.

\bibitem[KS]{kom-sim}
J.~Koml{\'o}s and M.~Simonovits.
\newblock {Szemer\'edi's regularity lemma and its applications in graph
  theory}.
\newblock In {\em Combinatorics, {P}aul {E}rd{\H o}s is eighty, {V}ol.\ 2
  ({K}eszthely, 1993)}, volume~2 of {\em Bolyai Soc. Math. Stud.}, pages
  295--352. J\'anos Bolyai Math. Soc., Budapest, 1996.

\bibitem[MS]{ma-sh}
M.~Malliaris and S.~Shelah.
\newblock {Regularity lemmas for stable graphs}.
\newblock {\em Trans. Amer. Math. Soc.} {\bf 366}(2014), 1551--1585.

\bibitem[MM]{masur-minsky2}
H.~A. Masur and Y.~N. Minsky.
\newblock {Geometry of the complex of curves. {II}. {H}ierarchical structure}.
\newblock {\em Geom. Funct. Anal.} {\bf 10}(2000), 902--974.

\bibitem[MS]{myasnikov-shumyatsky}
Alexei Myasnikov and Pavel Shumyatsky.
\newblock {Discriminating groups and c-dimension}.
\newblock {\em J. Group Theory} {\bf 7}(2004), 135--142.

\bibitem[Pil]{pillay}
Anand Pillay.
\newblock {\em An introduction to stability theory}, volume~8 of {\em Oxford
  Logic Guides}.
\newblock The Clarendon Press Oxford University Press, New York, 1983.



\bibitem[RS]{roberston-seymour}
Neil Robertson and P.~D. Seymour.
\newblock {Graph minors. {XX}. {W}agner's conjecture}.
\newblock {\em J. Combin. Theory Ser. B} {\bf 92}(2004), 325--357.

\bibitem[Ser]{servatius}
Herman Servatius.
\newblock {Automorphisms of graph groups}.
\newblock {\em J. Algebra} {\bf 126}(1989), 34--60.

\bibitem[Sze]{szemeredi}
E.~Szemer{\'e}di.
\newblock {On sets of integers containing no {$k$} elements in arithmetic
  progression}.
\newblock {\em Acta Arith.} {\bf 27}(1975), 199--245.
\newblock Collection of articles in memory of Juri{\u\i} Vladimirovi{\v{c}}
  Linnik.

\bibitem[TZ]{te-zi}
Katrin Tent and Martin Ziegler.
\newblock {\em A course in model theory}, volume~40 of {\em Lecture Notes in
  Logic}.
\newblock Association for Symbolic Logic, La Jolla, CA, 2012.

\end{thebibliography}
\end{document}